\documentclass[a4paper,reqno]{amsart}
\usepackage{amssymb, amsmath, amscd}

\def\gronko{\vphantom{\vrule height 10pt}}
\newtheorem{theorem}{Theorem}[section]
\newtheorem{definition}[theorem]{Definition}
\newtheorem{lemma}[theorem]{Lemma}

\newtheorem{remark}[theorem]{Remark}

\def\mapright#1{\smash{\mathop{\longrightarrow}\limits\sp{#1}}}
\makeatletter
 \@addtoreset{equation}{section}
\makeatother
\begin{document}
\title[Geometric realizations]
 {Geometric realizations of affine K\"ahler curvature models}
\author[M. Brozos-V\'azquez et. al.]{M. Brozos-V\'azquez,
P. Gilkey, and S. Nik\v cevi\'c}
\address{MBV: Department of Mathematics, E. U. Polit\'ecnica, Universidade da Coru\~na, Spain\\
E-mail: mbrozos@udc.es}
\address{PG: Mathematics Department, University of Oregon\\
  Eugene OR 97403 USA\\
  E-mail: gilkey@uoregon.edu}
\address{SN: Mathematical Institute, Sanu,
Knez Mihailova 36, p.p. 367,
11001 Belgrade,
Serbia.\\ Email: stanan@mi.sanu.ac.rs}
\begin{abstract}{We show that every K\"ahler affine curvature model can be realized geometrically.}\end{abstract}
\maketitle
\centerline{\bf Dedicated to Heinrich Wefelscheid}\centerline{\bf on the occasion of his 70th birthday}
\section{Introduction}
The study of curvature is central in modern differential geometry. 
One often first considers a
problem abstractly in a purely algebraic framework and then subsequently passes to the geometrical context
by examining the appropriate geometric realization problem.

The nature of the realization question of course depends on the problem under consideration and its
context. Among the realization problems many of them arise in very natural fashions. We
shall work in the affine setting and consider the curvature associated to a K\"ahler
affine connection. This curvature operator satisfies some known universal symmetries (see
Equations \eqref{eqn-1.a} and
\eqref{eqn-1.i} below). We study the inverse problem. We shall consider a tensor
which satisfies those symmetries and show that it is indeed the curvature tensor associated to an affine
connection (Theorem
\ref{thm-1.8}). We now briefly survey other
results in this field to put our result in context.

\subsection{Affine geometry} The pair $(M,\nabla)$ is said to be an {\it affine manifold} if $\nabla$ is a
torsion free connection on the tangent bundle
$TM$ of a smooth $m$-dimensional manifold $M$. Let
$$\mathcal{R}(x,y):=\nabla_x\nabla_y-\nabla_y\nabla_x-\nabla_{[x,y]}$$
be the associated {\it curvature operator} which has the symmetries:
\begin{equation}\label{eqn-1.a}
\mathcal{R}(x,y)=-\mathcal{R}(y,x)\quad\text{and}\quad
\mathcal{R}(x,y)z+\mathcal{R}(y,z)x+\mathcal{R}(z,x)y=0\,.
\end{equation}

It is convenient to work in a purely algebraic context. Let $V$ be a real $m$-dimensional
vector space. We say that
$\mathcal{A}\in\operatorname{End}(V)\otimes V^*\otimes V^*$ is a {\it generalized curvature
operator} if it satisfies the symmetries of Equation (\ref{eqn-1.a}). Let
$\mathfrak{A}=\mathfrak{A}(V)$ be the linear vector space of all generalized curvature tensors on $V$. If
$\mathcal{A}\in\mathfrak{A}$,  the structure
$(V,\mathcal{A})$ is said to be an {\it affine curvature model}. We say that an affine curvature
model $(V,\mathcal{A})$ is {\it geometrically realizable by an affine manifold} if there exists an affine manifold
$(M,\nabla)$, a point
$P\in M$, and an isomorphism
$\Xi:V\rightarrow T_PM$ so that
$\Xi^*\mathcal{R}_P=\mathcal{A}$. The following result
\cite{BNGS06} permits one to pass from the algebraic to the geometric context. It shows that the relations of Equation
(\ref{eqn-1.a}) generate the universal symmetries of
the curvature operator associated to a torsion free connection:

\begin{theorem}\label{thm-1.1}
Every affine curvature model is geometrically realizable by an affine manifold.
\end{theorem}

\subsection{Weyl geometry} We consider a {\it mixed structure}
$(M,g,\nabla)$ where
$g$ is a Riemannian metric on $M$ and where $\nabla$ is a torsion free connection on $M$. We use the metric
to lower indices and define $R\in\otimes^4T^*M$ by setting:
$$R(x,y,z,w)=g(\mathcal{R}(x,y)z,w)\,.$$
The symmetries of Equation (\ref{eqn-1.a}) then become:
\begin{equation}\label{eqn-1.b}
\begin{array}{l}
R(x,y,z,w)+R(y,x,z,w)=0,\\
R(x,y,z,w)+R(y,z,x,w)+R(z,x,y,w)=0\,.
\end{array}\end{equation}
There are two different traces it is natural to consider in this setting. Let $\{e_i\}$ be a local orthonormal frame
for $TM$. We adopt the {\it Einstein convention} and sum over repeated indices to define:
\begin{equation}\label{eqn-1.c}
\rho_{13}(x,y):=R(e_i,x,e_i,y)\quad\text{and}\quad\rho_{14}(x,y):=R(e_i,x,y,e_i)\,.
\end{equation}
The {\it scalar curvature $\tau$} is then defined by setting:
$$\tau(A):=\rho_{14}(e_i,e_i)=-\rho_{13}(e_i,e_i)\,.$$
The tensor $\rho_{13}$, the tensor $\rho_{14}$, and the scalar curvature $\tau$ are all independent of the particular
orthonormal frame which is chosen.
Since one may also express $\rho_{14}(x,y)=\operatorname{Tr}(z\rightarrow\mathcal{R}(z,x)y)$, $\rho_{14}$ is defined in
the affine setting. By contrast, $\rho_{13}$ and $\tau$ are dependent on the metric.

The mixed structure $(M,g,\nabla)$ is
said to be a {\it Weyl manifold} \cite{W22} if there exists a smooth $1$-form $\phi$ so that
$$\nabla g=-2\phi\otimes g\,.$$
If $(M,g,\nabla)$ is a Weyl manifold, then there is an additional curvature symmetry which is satisfied
(see \cite{H94}):
\begin{equation}\label{eqn-1.d}
R(x,y,z,w)+R(x,y,w,z)=\textstyle\frac2m\{\rho_{14}(y,x) - \rho_{14}(x,y)\}g(z,w)\,.
\end{equation}

We fix a positive definite metric $\langle\cdot,\cdot\rangle$ on $V$ henceforth. We say that $A\in\mathfrak{A}$ is
a {\it Weyl generalized curvature tensor} if in addition it satisfies the compatibility relationship of Equation
(\ref{eqn-1.d}); let $\mathfrak{W}=\mathfrak{W}(V,\langle\cdot,\cdot\rangle)$ be the space of all such tensors. If
$A\in\mathfrak{W}$, then the structure
$(V,\langle\cdot,\cdot\rangle,\nabla,A)$ is said to be a {\it Weyl curvature model}. Such
a structure is said to be {\it geometrically realizable by a Weyl manifold} if there exists
a Weyl manifold $(M,g,\nabla)$, a point $P\in M$, and an isomorphism $\Xi$ from
$V$ to $T_PM$ so that
$\Xi^*g_P=\langle\cdot,\cdot\rangle$ and so that $\Xi^*R_P=A$. The following result
\cite{GNS10} extends Theorem
\ref{thm-1.1} to the Weyl setting; it shows that Equations  (\ref{eqn-1.b}) and (\ref{eqn-1.d}) generate the universal
symmetries of the curvature tensor in Weyl geometry:

\begin{theorem}\label{thm-1.2}
Every Weyl curvature model is geometrically realizable by a Weyl manifold.
\end{theorem}

\subsection{Riemannian geometry} If $g$ is a Riemannian metric on $M$, let $\mathcal{R}^g$ be the curvature operator and
let $R^g$ be the curvature tensor which are defined by the Levi-Civita connection $\nabla^g$. There is an additional
curvature symmetry in this case:
\begin{equation}\label{eqn-1.e}
R^g(x,y,z,w)=R^g(z,w,x,y)\,.
\end{equation}
One says that $A\in\mathfrak{A}$ is a {\it Riemannian curvature
tensor} if in addition it satisfies the symmetry of Equation (\ref{eqn-1.e}).  Let
$\mathfrak{R}=\mathfrak{R}(V,\langle\cdot,\cdot\rangle)$ be the space of all such tensors. 
If $R\in\mathfrak{R}$, then Equations (\ref{eqn-1.b}) and (\ref{eqn-1.e}) imply the additional symmetry:
$$R(x,y,z,w)=-R(x,y,w,z)\,.$$ 
It now follows that:
$$
\mathfrak{R}\subset\mathfrak{W}\subset\mathfrak{A}\,.$$
Thus Weyl geometry is in a certain sense intermediate between affine and Riemannian geometry. If
$A\in\mathfrak{R}$, then the structure $(V,\langle\cdot,\cdot\rangle,A)$ is said to be a {\it Riemannian
curvature model}. Such a structure is said to be {\it geometrically realizable by a Riemannian manifold} if
there exists a Riemannian manifold $(M,g)$, a point
$P\in M$, and an isomorphism $\Xi:V\rightarrow T_PM$ so that $\Xi^*g_P=\langle\cdot,\cdot\rangle$ and so that
$\Xi^*R_P^g=A$. The following result is well
known - see, for example, the discussion in
\cite{BNGS06}. It shows that the universal symmetries of
the curvature tensor associated to the Levi-Civita connection are given by Equations
(\ref{eqn-1.a}) and (\ref{eqn-1.e}):
\begin{theorem}\label{thm-1.3}
Every Riemannian curvature model can be realized geometrically by a Riemannian manifold.
\end{theorem}

\subsection{Almost Hermitian geometry} Let $m=2\bar m$ be even henceforth. A triple $(M,g,J)$ is said to be an {\it
almost Hermitian manifold} if
$g$ is a Riemannian metric on $M$ and if
$J$ is an endomorphism of $TM$ satisfying $J^*g=g$ and $J^2=-\operatorname{Id}$. One says an endomorphism $J$
of $V$ {\it gives $V$ a complex structure} if $J^2=-\operatorname{Id}$. A quadruple
$(V,\langle\cdot,\cdot\rangle,J,A)$ is said to be an {\it almost Hermitian curvature model} if $J$ gives $V$ a
complex structure, if
$J^*\langle\cdot,\cdot\rangle=\langle\cdot,\cdot\rangle$, and if $A\in\mathfrak{R}$. One says such a structure
can be {\it geometrically realized} by an almost Hermitian manifold if there exists an almost Hermitian
manifold $(M,g,J)$, a point $P\in M$, and an isomorphism $\Xi:V\rightarrow V$ so that
$\Xi^*g_P=\langle\cdot,\cdot\rangle$, $\Xi^*J_P=J$, and $ \Xi^*R_P^g=A$. There are no
additional curvature restrictions and one has the following result \cite{GBKNW08}:
\begin{theorem}\label{thm-1.4}
Every almost Hermitian curvature model can be realized geometrically by an almost Hermitian manifold.
\end{theorem}

\subsection{Hermitian geometry} If $(M,g,J)$ is an almost Hermitian manifold, then the almost complex structure $J$
is said to be {\it integrable} and $(M,g,J)$ is called a {\it Hermitian manifold} if there are local charts
$(x_1,...,x_{\bar m},y_1,...,y_{\bar m})$ so that

\begin{equation}\label{eqn-1.f}
J\partial_{x_i}=\partial_{y_i}\quad\text{and}\quad J\partial_{y_i}=-\partial_{x_i}\,.
\end{equation}
Equivalently, via the
Newlander--Nirenberg Theorem \cite{NN57}, this means that the {\it Nijenhuis tensor}
$$N_{J}(x,y):=[x,y]+J[Jx,y]
 +J[x,Jy]-[Jx,Jy]$$
vanishes. There is an extra
curvature restriction in this setting discovered by Gray
\cite{gray}:
\begin{eqnarray}
0&=&R^g(x,y,z,w)+R^g(Jx,Jy,Jz,Jw)-R^g(Jx,Jy,z,w)\nonumber\\
&-&R^g(x,y,Jz,Jw)-R^g(Jx,y,Jz,w)-R^g(x,Jy,z,Jw)\label{eqn-1.g}\\
&-&R^g(Jx,y,z,Jw)-R^g(x,Jy,Jz,w)\,.\nonumber
\end{eqnarray}
One says that an almost Hermitian curvature model $(V,\langle\cdot,\cdot\rangle,J,A)$ is a {\it Hermitian curvature
model} if $A$ also satisfies Equation (\ref{eqn-1.g}). One has \cite{BGKN09}:
\begin{theorem}\label{thm-1.5}
Every Hermitian curvature model is geometrically realizable by a Hermitian manifold.
\end{theorem}

\subsection{Riemannian K\"ahler geometry} The {\it K\"ahler} form of an almost Hermitian manifold
$(M,g,J)$ is defined by setting $\Omega(x,y):=g(x,Jy)$. One says that an almost Hermitian manifold is a {\it K\"ahler}
manifold if $J$ is integrable and $d\Omega=0$ or, equivalently, if
$\nabla ( J)=0$. This implies a restriction on
curvature:
\begin{equation}\label{eqn-1.h}
R^g(x,y,z,w)=R^g(x,y,Jz,Jw)\,.
\end{equation}
An almost Hermitian curvature model $(V,\langle\cdot,\cdot\rangle,J,A)$ is said to be a {\it K\"ahler
curvature model} if Equation (\ref{eqn-1.h}) is satisfied; as necessarily Equation (\ref{eqn-1.g}) is
satisfied in this setting, any K\"ahler curvature model is a Hermitian curvature model.  One
has
\cite{BGM09} that:
\begin{theorem}\label{thm-1.6}
Any K\"ahler curvature model is geometrically
realizable by a K\"ahler manifold.
\end{theorem}

\begin{remark}\rm
In fact, more is true. One can use the Cauchy-Kovalevskaya Theorem (see, for example, the
discussion in Evans \cite{E}) to show that the realizations in Theorems
\ref{thm-1.2}-\ref{thm-1.6} can be chosen to have constant scalar curvature. The arguments work equally in
the pseudo-Riemannian setting. Furthermore, Theorem \ref{thm-1.4}, Theorem \ref{thm-1.5}, and Theorem
\ref{thm-1.6} can be extended to the almost para-Hermitian setting, the para-Hermitian setting, and the
para-K\"ahler setting, respectively, \cite{GBKNW08,BGM09,BGNV09}.
\end{remark}

\subsection{Affine K\"ahler geometry} We now return to the affine setting. One says $(M,J,\nabla)$ is an {\it affine
K\"ahler manifold} if $J$ is an integrable almost complex structure on $M$, if $\nabla$
is a torsion free connection on $M$, and if $\nabla (J)=0$. This then implies the
curvature symmetry
\begin{equation}\label{eqn-1.i}
\mathcal{R}(x,y)J=J\mathcal{R}(x,y)\,.
\end{equation}
Let $J$ give $V$ a complex structure and let $\mathfrak{K}=\mathfrak{K}(V,J)\subset\mathfrak{A}$ be the
subspace of generalized algebraic curvature operators satisfying Equation (\ref{eqn-1.i}). If
$\mathcal{A}\in\mathfrak{K}$, then the triple
$(V,J,\mathcal{A})$ is said to be a {\it K\"ahler affine curvature model}. We say such a structure is {\it
geometrically realizable by an affine K\"ahler manifold} if there exists an affine K\"ahler manifold
$(M,J,\nabla)$, a point $P\in M$, and an isomorphism $\Xi:V\rightarrow T_PM$ so $\Xi^*\mathcal{R}=\mathcal{A}$
and $\Xi^*J_P=J$. The following is the main result of this paper and generalizes the geometric realization
results described above to this setting:
\begin{theorem}\label{thm-1.8}
Every K\"ahler affine curvature model is geometrically realizable by an affine K\"ahler manifold.
\end{theorem}

\subsection{Outline of the paper} Theorems of this type usually use curvature decomposition theory. For example, the
proof of Theorem \ref{thm-1.2} rests upon work of Higa \cite{H94}. The proof of  Theorem
\ref{thm-1.5} and the proof of Theorem
\ref{thm-1.6} rest upon the decomposition of Tricerri and Vanhecke \cite{TV81}. The proof of Theorem \ref{thm-1.8}
will rest upon the decomposition of
$\mathfrak{K}(V,J)$ as a unitary module given in \cite{PN91,N89,N94}. We shall review this decomposition in Section
\ref{sect-2}. In Section \ref{sect-3}, we let $J$ be the standard complex structure on $\mathbb{R}^m$ 
of Equation (\ref{eqn-1.f})
and construct affine K\"ahler connections. The curvature decomposition of Section \ref{sect-2} and
the construction of Section \ref{sect-3} will then be applied in Section
\ref{sect-4} to complete the proof of Theorem
\ref{thm-1.8}.

\section{The decomposition of $\mathfrak{K}(V,J)$ as a unitary module}\label{sect-2}
Let $J$ define a complex structure on $V$. Introduce an auxiliary positive definite inner product
$\langle\cdot,\cdot\rangle$ so $J^*\langle\cdot,\cdot\rangle=\langle\cdot,\cdot\rangle$. The associated unitary group is defined by
$$\mathcal{U}:=\{T\in\operatorname{GL}(V):T^*\langle\cdot,\cdot\rangle=\langle\cdot,\cdot\rangle\quad\text{and}
\quad TJ=JT\}\,.$$
Pullback makes $\mathfrak{K}(V,J)$ into a $\mathcal{U}$ module. In this section, we shall describe the decomposition of
$\mathfrak{K}$ into $\mathcal{U}$ modules given in \cite{PN91,N89}.

\subsection{Ricci tensors}
We may use the inner product to identify $\mathfrak{K}$ with
the set of all $A\in\otimes V^*$ so that:
\begin{eqnarray*}
&&A(x,y,z,w)=-A(y,x,z,w),\\
&&A(x,y,z,w)+A(y,z,x,w)+A(z,x,y,w)=0,\\
&&A(x,y,z,w)=A(x,y,Jz,Jw)\,.
\end{eqnarray*}

We may decompose
\begin{eqnarray*}
&&\mathfrak{K}^\pm:=\left\{\mathcal{A}\in\mathfrak{K}:\mathcal{A}(Jx,Jy)=\pm\mathcal{A}(x,y)\right\}\\
&&\qquad=\{A\in\mathfrak{K}:A(Jx,Jy,Jz,Jw)=\pm A(x,y,z,w)\}\,.
\end{eqnarray*}
 We begin by deriving some basic facts concerning
the Ricci tensors $\rho_{13}$ and $\rho_{14}$ defined in Equation (\ref{eqn-1.c}):

\begin{lemma}\label{lem-2.1}
\
\begin{enumerate}
\item If $A\in\mathfrak{K}$, then $\rho_{13}(x,y)=\rho_{13}(Jx,Jy)$.\item If $A\in\mathfrak{K}^+$, then $J^*\rho_{13}=\rho_{13}$ and $J^*\rho_{14}=\rho_{14}$.\item If $A\in\mathfrak{K}^-$, then $\rho_{13}=0$ and $J^*\rho_{14}=-\rho_{14}$.
\end{enumerate}
\end{lemma}

\begin{proof} Let $\{e_i\}$ be an orthonormal basis for $V$; $\{Je_i\}$ also is an orthonormal basis for $V$.
Let $A\in\mathfrak{K}$.
We sum over $i$ and compute:
\medbreak\quad
$2\rho_{13}(Jx,Jy)=A(e_i,Jx,e_i,Jy)+A(Je_i,Jx,Je_i,Jy)$
\smallbreak\quad
$\phantom{2\rho_{13}(Jx,Jy)}=-A(e_i,Jx,Je_i,y)+A(Je_i,Jx,e_i,y)$
\smallbreak\quad
$\phantom{2\rho_{13}(Jx,Jy)}=A(Jx,Je_i,e_i,y)+A(Je_i,e_i,Jx,y)-A(Jx,Je_i,e_i,y)$
\smallbreak\quad
$\phantom{2\rho_{13}(Jx,Jy)}=A(Je_i,e_i,Jx,y)=-A(Je_i,e_i,x,Jy)$
\smallbreak\quad
$\phantom{2\rho_{13}(Jx,Jy)}=A(e_i,x,Je_i,Jy)+A(x,Je_i,e_i,Jy)$
\smallbreak\quad
$\phantom{2\rho_{13}(Jx,Jy)}=A(e_i,x,e_i,y)+A(Je_i,x,Je_i,y)$
\smallbreak\quad
$\phantom{2\rho_{13}(Jx,Jy)}=2\rho_{13}(x,y)$
\medbreak\noindent
This proves Assertion (1).
Let $A\in\mathfrak{K}^\pm$. We compute:
\medbreak\quad
$\rho_{14}(Jx,Jy)=A(e_i,Jx,Jy,e_i)=\pm A(Je_i,JJx,JJy,Je_i)$
\smallbreak\qquad
$=\pm A(Je_i,x,y,Je_i)=\pm \rho_{14}(x,y)$,
\smallbreak\quad
$\rho_{13}(Jx,Jy)=A(e_i,Jx,e_i,Jy)=\pm  A(Je_i,JJx,Je_i,JJy)$
\smallbreak\qquad
$=\pm  A(Je_i,x,Je_i,y)=\pm \rho_{13}(x,y)$.
\medbreak\noindent Assertions (2) and (3) follow from these computations and from Assertion (1).
\end{proof}
\subsection{Linear and quadratic representations} The trivial representation appears with multiplicity $2$ in
$\mathfrak{K}$. We define two scalar invariants by setting:
$$\tau:=A(e_i,e_j,e_j,e_i)\quad\text{and}\quad
 \check\tau_J:=A(e_i,Je_j,e_j,e_i)\,.
$$
Decompose $\otimes^2V^*=\Lambda^2\oplus S^2$ as the direct sum of the alternating and the symmetric
$2$-tensors. We can use $J$ and the metric to further decompose
$S^2=S^2_-\oplus\mathbb{R}\oplus S^2_{0,+}$ and
$\Lambda^2=\Lambda^2_-\oplus\mathbb{R}\oplus\Lambda^2_{0,+}$ into irreducible inequivalent $\mathcal{U}$ modules where:
\begin{eqnarray*}
&&S^2_-:=\{\theta\in\otimes^2V^*:J^*\theta=-\theta,\quad\theta(x,y)=\theta(y,x)\},\\
&&S^2_{0,+}:=\{\theta\in\otimes^2V^*:J^*\theta=\theta,\quad\theta(x,y)=\theta(y,x),\quad
\theta\perp\langle\cdot,\cdot\rangle\},\\
&&\Lambda^2_-:=\{\theta\in\otimes^2V^*:J^*\theta=-\theta,\quad\theta(x,y)=-\theta(y,x)\},\\
&&\Lambda^2_{0,+}:=\{\theta\in\otimes^2V^*:J^*\theta=\theta,\quad\theta(x,y)=-\theta(y,x),\quad\theta\perp\Omega\}\,.
\end{eqnarray*}
The following Lemma will follow from computations we shall perform in Section
\ref{sect-3}:

\begin{lemma}\label{lem-2.2}
\ \begin{enumerate}
\item $\tau\oplus\check\tau_J:\mathfrak{K}^+\rightarrow\mathbb{R}\oplus\mathbb{R}\rightarrow0$.
\item $\rho_{14}:\mathfrak{K}^-\rightarrow S_-^2\oplus\Lambda_-^2\rightarrow0$.
\item $\rho_{14}\oplus\rho_{13}:\mathfrak{K}^+\rightarrow
S_{0,+}^2\oplus\Lambda_{0,+}^2\oplus S_{0,+}^2\oplus\Lambda_{0,+}^2\rightarrow0$.
\end{enumerate}\end{lemma}

All short exact sequences of unitary modules split naturally. Let $\rho_{ij}^\pm:=\rho_{ij}$ restricted to
$\mathfrak{K}^\pm$. We use Lemma \ref{lem-2.2} to define 8 mutually orthogonal submodules of
$\mathfrak{K}$ by requiring that:
\begin{eqnarray*}
&&\tau\oplus\check\tau_J:W_5\oplus W_6\mapright{\approx}\mathbb{R}\oplus\mathbb{R},\\
&&\rho_{14}^-:W_2\oplus W_4\mapright{\approx}S_-^2\oplus\Lambda_-^2,\\
&&\rho_{14}^+\oplus\rho_{13}^+:W_1\oplus W_3\oplus W_7\oplus W_8\mapright{\approx}
S_{0,+}^2\oplus\Lambda_{0,+}^2\oplus S_{0,+}^2\oplus\Lambda_{0,+}^2\,.
\end{eqnarray*}
\subsection{Other representations}
There are 4 other submodules of $\mathfrak{K}$ which appear in the decomposition of \cite{PN91,N89}. Let:
\begin{eqnarray*}
&&W_9:=\{A\in\mathfrak{K}^+:A(x,y,z,w)=-A(x,y,w,z)\}\cap\ker(\rho_{14}),\\
&&W_{10}:=\{A\in\mathfrak{K}^+:A(x,y,z,w)=A(x,y,w,z)\}\cap\ker(\rho_{14}),\\
&&W_{11}:=\mathfrak{K}^+\cap W_9^\perp\cap W_{10}^\perp\cap\ker(\rho_{13})\cap\ker(\rho_{14}),\\
&&W_{12}:=\mathfrak{K}^-\cap\ker(\rho_{14})\,.
\end{eqnarray*}
One then has the fundamental result \cite{PN91,N89}:
\begin{theorem}\label{thm-2.3}
\ \begin{enumerate}
\item If $\dim(V)=4$, then $\mathfrak{K}=W_1\oplus...\oplus W_{10}$, $W_{11}=\{0\}$, and $W_{12}=\{0\}$. The spaces
$W_1$, ...,  $W_{10}$ are mutually orthogonal irreducible $\mathcal{U}$ modules.
\item If $\dim(V)\ge6$, then $\mathfrak{K}=W_1\oplus...\oplus W_{12}$. The spaces $W_1$, ...,  $W_{12}$
are mutually orthogonal irreducible $\mathcal{U}$ modules.
\end{enumerate}\end{theorem}
\begin{remark}\label{rmk-2.4}
\rm The dimensions of these modules are given by:
\medbreak\quad
$$\begin{array}{|l|l|l|l|}\noalign{\hrule}
W_1&\bar m^2-1&W_7&\bar m^2-1\gronko\\\noalign{\hrule}
W_2&\bar m(\bar m+1)&W_8&\bar m^2-1\gronko\\\noalign{\hrule}
W_3&\bar m^2-1&W_9&\frac14\bar m^2(\bar m-1)(\bar m+3)\gronko\\\noalign{\hrule}
W_4&\bar m(\bar m-1)&W_{10}&\frac14\bar m^2(\bar m-1)(\bar m+3)\gronko\\\noalign{\hrule}
W_5&1&W_{11}&\frac12(\bar m-1)(\bar m+1)(\bar m-2)(\bar m+2)\gronko\\\noalign{\hrule}
W_6&1&W_{12}&\frac23\bar m^2(\bar m-2)(\bar m+2)\gronko\\\noalign{\hrule}
\end{array}$$
Clearly $W_5\approx W_6$. Let $T\Theta(x,y):=\Theta(x,Jy)$. The map $\Theta\rightarrow T\Theta$ induces an
isomorphism between
$\Lambda_{0,+}$ and $S_{0,+}$. Consequently
$W_1\approx W_3\approx W_7\approx W_8$. Similarly the correspondence $TA(x,y,z,w)=A(x,y,z,Jw)$ defines an
isomorphism
$W_9\approx W_{10}$. However otherwise these modules are distinct on dimensional grounds where, if
$n=4$, we ignore $W_{11}$ and $W_{12}$. We note
$$\dim\mathfrak{K}=\textstyle\frac13\bar m^2(\bar m+1)(5\bar m-2)\,.$$
\end{remark}

\section{Constructing K\"ahler connections}\label{sect-3}

\begin{definition}\label{defn-3.1}
\rm Let $J$ be the canonical complex structure on $\mathbb{R}^{m}=\mathbb{C}^{\bar m}$ of {\rm Equation (\ref{eqn-1.f})}.
Let indices
$i,j,k$ range from $1$ to
$\bar m$. Let $e_i:=\partial_{x_i}$ and $f_i:=\partial_{y_i}=Je_i$. Let $\{u_{ijk},v_{ijk}\}$ be
smooth functions on $\mathbb{R}^{m}$ where $u_{ijk}=u_{jik}$ and $v_{ijk}=v_{jik}$. Let $\Theta:=u+\sqrt{-1}v$. Define a torsion
free connection $\nabla$ by defining:
\begin{eqnarray*}
&&\nabla_{e_i}e_j=-\nabla_{f_i}f_j=u_{ijk}e_k+v_{ijk}f_k,\\
&&\nabla_{f_i}e_j=\nabla_{e_i}f_j=-v_{ijk}e_k+u_{ijk}f_k\,.
\end{eqnarray*}
\end{definition}

The following simple observation will be crucial
to our study.

\begin{lemma}\label{lem-3.1}
\  \begin{enumerate}
\item $\nabla$ is torsion free and $\nabla(J)=0$.
\item If $\Theta$ is holomorphic, then $R\in\mathfrak{K}^-$.
\item If $\Theta$ is anti-holomorphic and if $\Theta(0)=0$, then $R(0)\in\mathfrak{K}^+$.
\end{enumerate}
\end{lemma}

\begin{proof} Since $u_{ijk}=u_{jik}$ and $v_{ijk}=v_{jik}$, $\nabla$ is torsion free. We may use $J$ to identify $\mathbb{R}^{m}$
with
$\mathbb{C}^{\bar m}$. Let $\nabla^c_{e_i}e_j=\Theta_{ijk}e_k$. If we extend $\nabla^c$ to be
complex linear, then $\nabla$ is the underlying real connection. Thus $\nabla J=J\nabla$ so the
connection is K\"ahler.

If $\Theta(0)=0$, then the curvature is given by $d\Theta$. By making a complex linear change of basis if necessary, we may assume
without loss of generality that the basis is orthonormal at the origin. Let
$A=R(0)$. We compute:
\medbreak\quad
$\begin{array}{ll}
A(e_i,e_j,e_k,e_l)=e_iu_{jkl}-e_ju_{ikl},&
A(e_i,e_j,f_k,f_l)=e_iu_{jkl}-e_ju_{ikl},\gronko\\
A(f_i,f_j,e_k,e_l)=-f_iv_{jkl}+f_jv_{ikl},&
A(f_i,f_j,f_k,f_l)=-f_iv_{jkl}+f_jv_{ikl},\gronko\\
A(e_i,e_j,e_k,f_l)=e_iv_{jkl}-e_jv_{ikl},&
A(e_i,e_j,f_k,e_l)=-e_iv_{jkl}+e_jv_{ikl},\gronko\\
A(f_i,f_j,e_k,f_l)=f_iu_{jkl}-f_ju_{ikl},&
A(f_i,f_j,f_k,e_l)=-f_iu_{jkl}+f_ju_{ikl},\gronko\\
A(e_i,f_j,e_k,e_l)=-e_iv_{jkl}-f_ju_{ikl},&
A(e_i,f_j,f_k,f_l)=-e_iv_{jkl}-f_ju_{ikl},\gronko\\
A(e_i,f_j,e_k,f_l)=e_iu_{jkl}-f_jv_{ikl},&
A(e_i,f_j,f_k,e_l)=-e_iu_{jkl}+f_jv_{ikl}.\gronko
\end{array}$
\medbreak\noindent If $\Theta$ is holomorphic, then $e_iu=f_iv$ and $e_iv=-f_iu$; it then follows $A\in\mathfrak{K}^-$. If
$\Theta$ is anti-holomorphic, then $e_iu=-f_iv$ and $e_iv=f_iu$; it then follows that $A\in\mathfrak{K}^+$. We complete the
proof by assuming that
$d\Theta(0)=0$ and by studying the quadratic terms. Let $R^c$ be the complex curvature. Then:
$$
R^c(e_i,e_j)e_k=\sum_{a,l}\{\Theta(e_i,e_a,e_l)\Theta(e_j,e_k,e_a)-\Theta(e_j,e_a,e_l)\Theta(e_i,e_k,e_a)\}\,.$$
From this it is clear that $R^c(f_i,f_j)=-R^c(e_i,e_j)$ and $R^c(e_i,f_j)=-R^c(Je_i,Jf_j)$. Disentangling the real and imaginary
parts of these operators then shows that $R\in\mathfrak{K}^-$.
\end{proof}
This construction provides K\"ahler connections whose curvature tensors lie in $\mathfrak{K}^+$ at a single point, but not globally.
It also provides K\"ahler connections whose curvature tensors always lie in $\mathfrak{K}^-$.

\section{The proof of Theorem \ref{thm-1.8}}\label{sect-4}
Let $\mathfrak{J}$ be the linear subspace of $\mathfrak{K}$ consisting of all curvature tensors which arise from the
construction given in Lemma \ref{lem-3.1} where $\Theta(0)=0$, and let
$\mathfrak{J}^+$ (resp. $\mathfrak{J}^-$) be the subspaces defined by $\Theta$ anti-holomorphic (resp. holomorphic). These are
clearly $\mathcal{U}$ sub-modules. We will complete the proof of Theorem \ref{thm-1.8} by showing
$\mathfrak{K}^\pm=\mathfrak{J}^\pm$. We begin by
establishing the following Lemma:

\begin{lemma}\
\begin{enumerate}
\item $\tau\oplus\check\tau_J:\mathfrak{J}^+\rightarrow\mathbb{R}\oplus\mathbb{R}\rightarrow0$.
\item $\rho_{14}:\mathfrak{J}^-\rightarrow S_-^2\oplus\Lambda_-^2\rightarrow0$.
\item $\rho_{14}\oplus\rho_{13}:\mathfrak{J}^+\rightarrow
S_{0,+}^2\oplus\Lambda_{0,+}^2\oplus S_{0,+}^2\oplus\Lambda_{0,+}^2\rightarrow0$.
\item $W_1\oplus W_3\oplus W_5\oplus W_6\oplus W_7\oplus
W_8\subset\mathfrak{J}^+$ and $W_2\oplus W_4\subset\mathfrak{J}^-$.
\end{enumerate}\end{lemma}

\begin{proof}We apply Lemma \ref{lem-3.1}. We establish Assertion (1) by
taking:
$$
\Theta_{111}=\varrho_1(x_1-\sqrt{-1}y_1)\quad\text{and}\quad\Theta_{122}=\Theta_{212}=\varrho_2(y_1+\sqrt{-1}x_1)\,.
$$
Let $\mathcal{A}:=\mathcal{R}(0)$. We may then establish Assertion (1) by computing:
\medbreak\quad$\begin{array}{ll}
\nabla_{e_1}e_1=-\nabla_{f_1}f_1=\varrho_1(x_1e_1-y_1f_1),&
\nabla_{e_1}f_1=\nabla_{f_1}e_1=\varrho_1(y_1e_1+x_1f_1),\gronko\\
\nabla_{e_1}e_2=-\nabla_{f_1}f_2=\varrho_2(y_1e_2+x_1f_2),&
\nabla_{e_2}e_1=-\nabla_{f_2}f_1=\varrho_2(y_1e_2+x_1f_2)\gronko,\\
\nabla_{e_1}f_2=\nabla_{f_1}e_2=\varrho_2(-x_1e_2+y_1f_2),&
\nabla_{e_2}f_1=\nabla_{f_2}e_1=\varrho_2(-x_1e_2+y_1f_2),\gronko
\end{array}$\smallbreak\quad$\begin{array}{ll}
\mathcal{A}(e_1,f_1)e_1=2\varrho_1f_1,\phantom{.........................}&
\mathcal{A}(e_1,f_1)f_1=-2\varrho_1e_1,\gronko\\
\mathcal{A}(e_1,e_2)e_1=-\mathcal{A}(e_1,f_2)f_1=\varrho_2f_2,&
\mathcal{A}(e_1,e_2)f_1=\mathcal{A}(e_1,f_2)e_1=-\varrho_2e_2,\gronko\\
\mathcal{A}(f_1,e_2)e_1=-\mathcal{A}(f_1,f_2)f_1=\varrho_2e_2,&
\mathcal{A}(f_1,e_2)f_1=\mathcal{A}(f_1,f_2)e_1=\varrho_2f_2,\gronko\\
\mathcal{A}(e_1,f_1)f_2=-2\varrho_2f_2,&
\mathcal{A}(e_1,f_1)e_2=-2\varrho_2e_2,\gronko
\end{array}$\smallbreak\quad$\begin{array}{ll}
\rho_{14}(e_1,e_1)=\rho_{14}(f_1,f_1)=-2\varrho_1,\phantom{......}&
\rho_{14}(e_1,f_1)=-\rho_{14}(f_1,e_1)=2\varrho_2,\gronko\\
\tau=-4\varrho_1,&\check\tau_J=-4\varrho_2.\end{array}$

\medbreak Next take $\Theta_{111}=\varrho_1(x_2+\sqrt{-1}y_2)$ and $\Theta_{222}=\varrho_2(x_1+\sqrt{-1}y_1)$. Again, let
$\mathcal{A}:=\mathcal{R}(0)$. Then:
\medbreak\quad$\begin{array}{ll}
\nabla_{e_1}e_1=-\nabla_{f_1}f_1=\varrho_1(x_2e_1+y_2f_1),&
\nabla_{e_1}f_1=\nabla_{f_1}e_1=\varrho_1(-y_2e_1+x_2f_1),\\
\nabla_{e_2}e_2=-\nabla_{f_2}f_2=\varrho_2(x_1e_2+y_1f_2),&
\nabla_{e_2}f_2=\nabla_{f_2}e_2=\varrho_2(-y_1e_2+x_1f_2),\gronko
\end{array}$\smallbreak\quad$\begin{array}{ll}
\mathcal{A}(e_2,e_1)e_1=-\mathcal{A}(e_2,f_1)f_1=\varrho_1e_1,\phantom{a}&
\mathcal{A}(f_2,e_1)e_1=-\mathcal{A}(f_2,f_1)f_1=\varrho_1f_1,\gronko\\
\mathcal{A}(e_2,e_1)f_1=\mathcal{A}(e_2,f_1)e_1=\varrho_1f_1,&
\mathcal{A}(f_2,e_1)f_1=\mathcal{A}(f_2,f_1)e_1=-\varrho_1e_1,\gronko\\
\mathcal{A}(e_1,e_2)e_2=-\mathcal{A}(e_1,f_2)f_2=\varrho_2e_2,&
\mathcal{A}(f_1,e_2)e_2=-\mathcal{A}(f_1,f_2)f_2=\varrho_2f_2,\gronko\\
\mathcal{A}(e_1,e_2)f_2=\mathcal{A}(e_1,f_2)e_2=\varrho_2f_2,&
\mathcal{A}(f_1,e_2)f_2=\mathcal{A}(f_1,f_2)e_2=-\varrho_2e_2,\gronko
\end{array}$\smallbreak\quad$\begin{array}{ll}
\rho_{14}(e_2,e_1)=-2\varrho_1,\qquad\qquad\qquad\phantom{......}&\rho_{14}(f_2,f_1)=2\varrho_1,\gronko\\
\rho_{14}(e_1,e_2)=-2\varrho_2,&\rho_{14}(f_1,f_2)=2\varrho_2.\gronko
\end{array}$
\medbreak\noindent
If we take $\varrho_1=\varrho_2$, then $\rho_{14}\in S_-^2$; if we take $\varrho_1=-\varrho_2$, then $\rho_{14}\in\Lambda_-^2$.
This proves Assertion
(2).

We begin the proof of Assertion (3) by taking:
\begin{eqnarray*}
&&\Theta_{111}=\varrho_1(x_1-\sqrt{-1}y_1)+\varrho_2(x_2-\sqrt{-1}y_2),\\
&&\Theta_{222}=\varrho_3(x_2-\sqrt{-1}y_2)+\varrho_4(x_1-\sqrt{-1}y_1)\,.
\end{eqnarray*}
We then have
\medbreak\qquad
$\nabla_{e_1}e_1=-\nabla_{f_1}f_1=(\varrho_1x_1+\varrho_2x_2)e_1-(\varrho_1y_1+\varrho_2y_2)f_1$,
\smallbreak\qquad
$\nabla_{e_1}f_1=\nabla_{f_1}e_1=(\varrho_1y_1+\varrho_2y_2)e_1+(\varrho_1x_1+\varrho_2x_2)f_1$,
\smallbreak\qquad
$\nabla_{e_2}e_2=-\nabla_{f_2}f_2=(\varrho_3x_2+\varrho_4x_1)e_2-(\varrho_3y_2+\varrho_4y_1)f_2$,
\smallbreak\qquad
$\nabla_{e_2}f_2=\nabla_{f_2}e_2=(\varrho_3y_2+\varrho_4y_1)e_2+(\varrho_3x_2+\varrho_4x_1)f_2$.
\medbreak\noindent Let $\mathcal{A}:=\mathcal{R}(0)$. Then:
\medbreak\quad$\begin{array}{ll}
\mathcal{A}(e_1,f_1)f_1=-2\varrho_1e_1,&
\mathcal{A}(e_1,f_1)e_1=2\varrho_1f_1,\gronko\\
\mathcal{A}(e_2,e_1)e_1=-\mathcal{A}(e_2,f_1)f_1=\varrho_2e_1,&
\mathcal{A}(e_2,e_1)f_1=\mathcal{A}(e_2,f_1)e_1=\varrho_2f_1,\gronko\\
\mathcal{A}(f_2,e_1)e_1=-\mathcal{A}(f_2,f_1)f_1=-\varrho_2f_1,&
\mathcal{A}(f_2,e_1)f_1=\mathcal{A}(f_2,f_1)e_1=\varrho_2e_1,\gronko\\
\mathcal{A}(e_2,f_2)f_2=-2\varrho_3e_2,&
\mathcal{A}(e_2,f_2)e_2=2\varrho_3f_2,\gronko\\
\mathcal{A}(e_1,e_2)e_2=-\mathcal{A}(e_1,f_2)f_2=\varrho_4e_2,&
\mathcal{A}(e_1,e_2)f_2=\mathcal{A}(e_1,f_2)e_2=\varrho_4f_2,\gronko\\
\mathcal{A}(f_1,e_2)e_2=-\mathcal{A}(f_1,f_2)f_2=-\varrho_4f_2,&
\mathcal{A}(f_1,e_2)f_2=\mathcal{A}(f_1,f_2)e_2=\varrho_4e_2,\gronko
\end{array}$
\smallbreak\quad$\begin{array}{ll}
\rho_{14}(e_1,e_1)=\rho_{14}(f_1,f_1)=-2\varrho_1,\qquad&
\rho_{14}(e_1,e_2)=\rho_{14}(f_1,f_2)=-2\varrho_4,\gronko\\
\rho_{14}(e_2,e_2)=\rho_{14}(f_2,f_2)=-2\varrho_3,&
\rho_{14}(e_2,e_1)=\rho_{14}(f_2,f_1)=-2\varrho_2,\gronko\\
\rho_{13}(e_1,e_1)=\rho_{13}(f_1,f_1)=2\varrho_1,&
\rho_{13}(e_1,e_2)=\rho_{13}(f_1,f_2)=0,\gronko\\
\rho_{13}(e_2,e_2)=\rho_{13}(f_2,f_2)=2\varrho_3,&
\rho_{13}(e_2,e_1)=\rho_{13}(f_2,f_1)=0.\gronko
\end{array}$
\medbreak\noindent Note that $\tau=-4\varrho_1-4\varrho_3$ and $\check\tau_J=0$.\begin{enumerate}
\item
Take $\vec\varrho=(0,1,0,1)$ to construct $\nabla$ with $0\ne\rho_{14}(A)\in S_{0,+}^2$
and $\rho_{13}=0$.
\item Take $\vec\varrho=(0,1,0,-1)$ to construct $\nabla$ with $0\ne\rho_{14}(A)\in\Lambda_{0,+}^2$ and $\rho_{13}=0$.
\item Take $\vec\varrho=(1,0,-1,0)$ to construct $\nabla$ with $0\ne\rho_{13}(A)\in S_{0,+}^2$.
\end{enumerate} Thus we complete the proof of
Assertion (3) by constructing an example where $\rho_{13}$ has a non-zero component in $\Lambda_{0,+}^2$. We take
$$\Theta_{122}=\Theta_{212}=\varrho_5(x_2-\sqrt{-1}y_2)\,.$$
Set $\mathcal{A}:=\mathcal{R}(0)$. We have:
\medbreak\quad$\begin{array}{ll}
\nabla_{e_1}e_2=-\nabla_{f_1}f_2=\varrho_5(x_2e_2-y_2f_2),&
\nabla_{e_1}f_2=\nabla_{f_1}e_2=\varrho_5(y_2e_2+x_2f_2),\gronko\\
\nabla_{e_2}e_1=-\nabla_{f_2}f_1=\varrho_5(x_2e_2-y_2f_2),&
\nabla_{f_2}e_1=\nabla_{e_2}f_1=\varrho_5(y_2e_2+x_2f_2),\gronko
\end{array}$
\smallbreak\quad$\begin{array}{ll}
\mathcal{A}(e_2,e_1)e_2=-\mathcal{A}(e_2,f_1)f_2=\varrho_5e_2,\phantom{.}&
\mathcal{A}(f_2,e_1)e_2=-\mathcal{A}(f_2,f_1)f_2=-\varrho_5f_2,\gronko\\
\mathcal{A}(e_2,e_1)f_2=\mathcal{A}(e_2,f_1)e_2=\varrho_5f_2,&
\mathcal{A}(f_2,e_1)f_2=\mathcal{A}(f_2,f_1)e_2=\varrho_5e_2,\gronko\\
\mathcal{A}(e_2,f_2)e_1=2\varrho_5f_2,&\mathcal{A}(e_2,f_2)f_1=-2\varrho_5e_2,\gronko
\end{array}$\smallbreak\quad$\begin{array}{ll}\varrho_{13}(e_1,e_2)=2\varrho_5,\qquad\qquad\qquad\qquad\phantom{.}&\varrho_{13}(f_1,f_2)=2\varrho_5.\gronko
\end{array}$
\medbreak\noindent
This belongs to $S_{0,+}^2\oplus\Lambda_{0,+}^2$. It is not symmetric and thus has a non-zero component in
$\Lambda_{0,+}^2$.  Assertion (3) follows.

Since the modules in question are irreducible, Assertions (1)-(3) show the maps of Lemma \ref{lem-2.2} define
surjective maps
\begin{eqnarray*}
&&\rho_{14}:\mathfrak{J}^-\rightarrow W_2\oplus W_4\rightarrow0,\\
&&\tau\oplus\check\tau_J\oplus\rho_{14}\oplus\rho_{13}:\mathfrak{J}^+\rightarrow W_1\oplus W_3\oplus W_5\oplus
W_6\oplus W_7\oplus W_8\rightarrow0\,.
\end{eqnarray*}
We consider the  collection of modules 
\begin{eqnarray*}
&&\mathcal{C}_1:=\{W_2,W_4\},\\
&&\mathcal{C}_2:=\{W_1,W_3,W_5,W_6,W_7,W_8\},\\
&&\mathcal{C}_3:=\{W_9,W_{10},W_{11},W_{12}\}
\end{eqnarray*}  where we omit $W_{11}$ and
$W_{12}$ if $m=4$. By Remark \ref{rmk-2.4}, no module from collection
$\mathcal{C}_i$ is isomorphic to any module from
$\mathcal{C}_j$ for
$i\ne j$ on dimensional grounds. Assertion (4) now follows from Theorem \ref{thm-2.3}.
\end{proof}

We complete the proof of Theorem \ref{thm-1.8} by establishing:
\begin{lemma}\label{lem-4.2}
\ \begin{enumerate}
\item If $m\ge6$, then $W_{12}\cap\mathfrak{J}^-\ne\{0\}$.
\item $W_{9}\cap\mathfrak{J}^+\ne\{0\}$.
\item $W_{10}\cap\mathfrak{J}^+\ne\{0\}$.
\item If $m\ge 6$, then $W_{11}\cap\mathfrak{J}^+\ne\{0\}$.
\end{enumerate}\end{lemma}

\begin{proof} Set $\Theta_{112}=x_3+\sqrt{-1}y_3$. We then have
\medbreak\quad
$\begin{array}{ll}
\nabla_{e_1}e_1=-\nabla_{f_1}f_1=x_3e_2+y_3f_2,&
\nabla_{e_1}f_1=\nabla_{f_1}e_1=-y_3e_2+x_3f_2,\gronko\\
\mathcal{A}(e_3,e_1)e_1=-\mathcal{A}(e_3,f_1)f_1=e_2,&
\mathcal{A}(f_3,e_1)e_1=-\mathcal{A}(f_3,f_1)f_1=f_2,\gronko\\
\mathcal{A}(e_3,e_1)f_1=\mathcal{A}(e_3,f_1)e_1=f_2,&
\mathcal{A}(f_3,e_1)f_1=\mathcal{A}(f_3,f_1)e_1=-e_2\,.\gronko
\end{array}$
\medbreak\noindent Since $\rho_{14}=0$ and $\mathcal{A}\in\mathfrak{J}^-$, $0\ne\mathcal{A}\in W_{12}$; this establishes Assertion
(1).

We clear the previous notation and take:
\medbreak\qquad
$\begin{array}{ll}
\Theta_{112}=\varrho_1(x_1-\sqrt{-1}y_1),&
\Theta_{111}=\varrho_3(x_2-\sqrt{-1}y_2),\\
\Theta_{121}=\Theta_{211}=\varrho_2(x_1-\sqrt{-1}y_1).\gronko
\end{array}$
\medbreak\noindent Consequently we have:
\medbreak\qquad
$\nabla_{e_1}e_1=-\nabla_{f_1}f_1=\varrho_1(x_1e_2-y_1f_2)+\varrho_3(x_2e_1-y_2f_1)$,
\smallbreak\qquad
$\nabla_{f_1}e_1=\nabla_{e_1}f_1=\varrho_1(y_1e_2+x_1f_2)+\varrho_3(y_2e_1+x_2f_1)$,
\smallbreak\qquad
$\nabla_{e_1}e_2=-\nabla_{f_1}f_2=\nabla_{e_2}e_1=-\nabla_{f_2}f_1=\varrho_2(x_1e_1-y_1f_1)$,
\smallbreak\qquad
$\nabla_{f_1}e_2=\nabla_{e_1}f_2=\nabla_{e_2}f_1=\nabla_{f_2}e_1=\varrho_2(y_1e_1+x_1f_1)$.
\medbreak\noindent Set $\mathcal{A}:=\mathcal{R}(0)$. Then:
\smallbreak\quad
$\begin{array}{ll}
\mathcal{A}(e_1,f_1)e_1=2\varrho_1 f_2,&
\mathcal{A}(e_1,f_1)f_1=-2\varrho_1 e_2,\gronko\\
\mathcal{A}(e_1,f_1)e_2=2\varrho_2 f_1,&
\mathcal{A}(e_1,f_1)f_2=-2\varrho_2 e_1,\gronko\\
\mathcal{A}(e_1,e_2)e_1=\varrho_2 e_1-\varrho_3 e_1,&
\mathcal{A}(e_1,e_2)f_1=\varrho_2 f_1-\varrho_3 f_1,\gronko\\
\mathcal{A}(e_1,f_2)e_1=\varrho_2 f_1+\varrho_3 f_1,&
\mathcal{A}(e_1,f_2)f_1=-\varrho_2 e_1-\varrho_3 e_1,\gronko\\
\mathcal{A}(f_1,f_2)e_1=\varrho_2 e_1-\varrho_3 e_1,&
\mathcal{A}(f_1,f_2)f_1=\varrho_2 f_1-\varrho_3f_1,\gronko\\
\mathcal{A}(f_1,e_2)e_1=-\varrho_2f_1-\varrho_3f_1,&\mathcal{A}(f_1,e_2)f_1=\varrho_2 e_1+\varrho_3 e_1.\gronko
\end{array}$
\medbreak\noindent We take $\vec\varrho=(-\frac12,-\frac12,-\frac12)$ to create an element $A_1\in\mathfrak{J}^+$ with:
\medbreak\qquad
$A_1(e_1,f_1,e_1,f_2)=A_1(f_1,e_1,f_1,e_2)=-1$,
\smallbreak\qquad $A_1(f_2,e_1,f_1,e_1)=A_1(e_2,f_1,e_1,f_1)=-1$,
\smallbreak\qquad
$A_1(e_1,f_1,f_2,e_1)=A_1(f_1,e_1,e_2,f_1)=1$,
\smallbreak\qquad $A_1(f_2,e_1,e_1,f_1)=A_1(e_2,f_1,f_1,e_1)=1$,
\smallbreak\qquad
$\rho_{14}(A_1)(e_1,e_2)=\rho_{14}(A_1)(e_2,e_1)=1,$
\smallbreak\qquad$\rho_{14}(A_1)(f_1,f_2)=\rho_{14}(A_1)(f_2,f_1)=1$.
\medbreak\noindent
Interchanging the roles of the indices `1' and `2' then creates a tensor
$A_2\in\mathfrak{J}^+$ such that
\medbreak\qquad
$A_2(e_2,f_2,e_2,f_1)=A_2(f_2,e_2,f_2,e_1)=-1$,
\medbreak\qquad
$A_2(f_1,e_2,f_2,e_2)=A_2(e_1,f_2,e_2,f_2)=-1$,
\smallbreak\qquad
$A_2(e_2,f_2,f_1,e_2)=A_2(f_2,e_2,e_1,f_2)=1$,
\smallbreak\qquad
$A_2(f_1,e_2,e_2,f_2)=A_2(e_1,f_2,f_2,e_2)=1$,
\smallbreak\qquad
$\rho_{14}(A_2)(e_1,e_2)=\rho_{14}(A_2)(e_2,e_1)=1$,
\smallbreak\qquad
$\rho_{14}(A_2)(f_1,f_2)=\rho_{14}(A_2)(f_2,f_1)=1$.\medbreak\noindent
These tensors are anti-symmetric in the last two indices so $\rho_{13}=-\rho_{14}$. We verify that $0\ne A_1-A_2\in
W_9\cap\mathfrak{J}^+$ which establishes Assertion (2).

Next, we take $\varrho=(\frac12,-\frac12,\frac12)$ to create a tensor such that:
\medbreak\qquad
$A_3(e_1,f_1,e_1,f_2)=A_3(e_1,f_1,f_2,e_1)=1$,
\smallbreak\qquad
$A_3(e_1,f_1,f_1,e_2)=A_3(e_1,f_1,e_2,f_1)=A_3(e_1,e_2,f_1,f_1)=-1$,
\smallbreak\qquad
$A_3(e_1,e_2,e_1,e_1)=A_3(f_1,f_2,e_1,e_1)=A_3(f_1,f_2,f_1,f_1)=-1$,
\smallbreak\qquad
$\rho_{14}(A_3)(e_1,e_2)=\rho_{14}(A_3)(f_1,f_2)=1,$
\smallbreak\qquad
$\rho_{14}(A_3)(e_2,e_1)=\rho_{14}(A_3)(f_2,f_1)=-1$.
\medbreak\noindent We interchange the roles of the indices ``1'' and ``2'' to create an element of $\mathfrak{J}^+$
 such that:
\medbreak\qquad
$A_4(e_2,f_2,e_2,f_1)=A_4(e_2,f_2,f_1,e_2)=1$,
\smallbreak\qquad
$A_4(e_2,f_2,f_2,e_1)=A_4(e_2,f_2,e_1,f_2)=A_4(e_2,e_1,f_2,f_2)=-1$,
\smallbreak\qquad
$A_4(e_2,e_1,e_2,e_2)=A_4(f_2,f_1,e_2,e_2)=A_4(f_2,f_1,f_2,f_2)=-1$,
\smallbreak\qquad
$\rho_{14}(A_3)(e_2,e_1)=\rho_{14}(A_3)(f_2,f_1)=1$,
\smallbreak\qquad
$\rho_{14}(A_3)(e_1,e_2)=\rho_{14}(A_3)(f_1,f_2)=-1$.
\medbreak\noindent These two tensors are symmetric in the last two indices so $\rho_{13}=\rho_{14}$. We then have $0\ne A_3+A_4\in
W_{10}\cap\mathfrak{J}^+$ which establishes Assertion (3).

We clear the previous notation and set $\Theta_{112}=x_3-\sqrt{-1}y_3$. Let
$\mathcal{A}=\mathcal{R}(0)$. Then:
\medbreak\quad
$\begin{array}{ll}
\nabla_{e_1}e_1=-\nabla_{f_1}f_1=x_3e_2-y_3f_2,&\nabla_{e_1}f_1=\nabla_{f_1}e_1=y_3e_2+x_3f_2,\\
\mathcal{A}(e_3,e_1)e_1=-\mathcal{A}(e_3,f_1)f_1=e_2,&\mathcal{A}(f_3,e_1)e_1=-\mathcal{A}(f_3,f_1)f_1=-f_2,\\
\mathcal{A}(e_3,e_1)f_1=\mathcal{A}(e_3,f_1)e_1=f_2,&\mathcal{A}(f_3,e_1)f_1=\mathcal{A}(f_3,f_1)e_1=e_2,\\
\rho_{13}(\mathcal{A})=\rho_{14}(\mathcal{A})=0.
\end{array}$
\medbreak\noindent We have $\mathcal{A}\in\mathfrak{J}^+\cap\ker(\rho_{13})\cap\ker(\rho_{14})=W_9\oplus W_{10}\oplus W_{11}$. To
prove Assertion (4), it suffices to show
$\mathcal{A}$ has a non-zero component in
$W_{11}$. Suppose to the contrary that $\mathcal{A}\in W_9\oplus W_{10}$. We may then decompose
$\mathcal{A}=\mathcal{A}_{9}+\mathcal{A}_{10}$ where $\mathcal{A}_9\in W_9$ and $\mathcal{A}_{10}\in W_{10}$. We lower indices to
define $A$, $A_9$, and $A_{10}$. We then have
\begin{eqnarray*}
A(x,y,z,w)+A(x,y,w,z)&=&A_9(x,y,z,w)-A_9(x,y,z,w)\\&+&A_{10}(x,y,z,w)+A_{10}(x,y,w,z)\\
&=&2A_{10}(x,y,z,w)\,.
\end{eqnarray*}
We now compute that:
$$A_{10}(f_3,f_1,e_2,e_1)+A_{10}(f_1,e_2,f_3,e_1)+A_{10}(e_2,f_3,f_1,e_1)=\textstyle\frac12+0+0\ne0\,.$$
This shows that the Bianchi identity; the Lemma now follows from this contradiction.\end{proof}

\begin{remark}
\rm In fact, we have proved just a bit more. We have shown that if $A\in\mathfrak{K}^-$, then $(V,J,A)$ is geometrically
realizable by an affine K\"ahler manifold
$(M,J,\nabla)$ where
$\mathcal{R}\in\mathfrak{K}^-$ at all points of $M$.\end{remark}

\section*{Acknowledgments} Research of M. Brozos-V\'azquez supported by projects MTM2009-07756 and INCITE09 207 151 PR (Spain).
Research of P. Gilkey partially supported by project MTM2009-07756 (Spain).
Research of S. Nik\v cevi\'c partially supported by project MTM2009-07756 (Spain) and by project 144032
(Serbia).

\end{document}